\documentclass[11pt,a4paper]{amsart}

\usepackage{amsmath}
\usepackage{amsfonts}
\usepackage{graphicx}
\usepackage{framed}
\usepackage{amssymb}
\usepackage{esvect}
\usepackage{xcolor}

\usepackage[backref=page]{hyperref}

\renewcommand*{\backref}[1]{}
\renewcommand*{\backrefalt}[4]{%
    \ifcase #1 (Not cited.)%
    \or        (Cited on page~#2.)%
    \else      (Cited on pages~#2.)%
    \fi}

\makeatletter
\@namedef{subjclassname@2020}{\textup{2020} Mathematics Subject Classification}
\makeatother

\usepackage{etex}
\usepackage{latexsym}
\usepackage[english]{babel}
\usepackage[utf8x]{inputenc}

\newcommand{\meanint}{-\!\!\!\!\!\!\!\int}

\def \R {\mathbb{R}}

\def \loc {\mathrm{loc}}
\def \de {\partial}
\def \LL {\mathcal{L}}

\usepackage{amsthm}
\theoremstyle{definition}
\newtheorem{definition}{Definition}[section]

\newtheorem{remark}[definition]{Remark}

\theoremstyle{plain}

\newtheorem{theorem}[definition]{Theorem}
\newtheorem{proposition}[definition]{Proposition}
\newtheorem{lemma}[definition]{Lemma}

\numberwithin{equation}{section}
\usepackage{miama}
\usepackage[T1]{fontenc}

\begin{document}

 \begin{abstract} Given a bounded open set~$\Omega\subseteq\R^n$,
 we consider the eigenvalue problem of a nonlinear mixed local/nonlocal operator
 with vanishing conditions in the complement of~$\Omega$.
 
 We prove that the second eigenvalue~$\lambda_2(\Omega)$ is always strictly larger than the first eigenvalue~$\lambda_1(B)$ of a ball~$B$ with volume half of that of~$\Omega$.
 
 This bound is proven to be sharp, by comparing to the limit case in which~$\Omega$
 consists of two equal balls far from each other. More precisely,
 differently from the local case, an optimal shape for the second eigenvalue problem does not exist, but a minimizing sequence is given by the union of two disjoint balls of half volume whose
mutual distance tends to infinity.
 \end{abstract}
 
 \title[Hong-Krahn-Szeg\"{o} for
 mixed operators]{A Hong-Krahn-Szeg\"{o} inequality \\ for mixed local and nonlocal operators}
 
 \author[S.\,Biagi]{Stefano Biagi}
 \author[S.\,Dipierro]{Serena Dipierro}
 \author[E.\,Valdinoci]{Enrico Valdinoci}
 \author[E.\,Vecchi]{Eugenio Vecchi}
 
 \address[S.\,Biagi]{Dipartimento di Matematica
 \newline\indent Politecnico di Milano \newline\indent
 Via Bonardi 9, 20133 Milano, Italy}
 \email{stefano.biagi@polimi.it}
 
 \address[S.\,Dipierro]{Department of Mathematics and Statistics
 \newline\indent University of Western Australia \newline\indent
 35 Stirling Highway, WA 6009 Crawley, Australia}
 \email{serena.dipierro@uwa.edu.au}
 
 \address[E.\,Valdinoci]{Department of Mathematics and Statistics
 \newline\indent University of Western Australia \newline\indent
 35 Stirling Highway, WA 6009 Crawley, Australia}
 \email{enrico.valdinoci@uwa.edu.au}
 
 \address[E.\,Vecchi]{Dipartimento di Matematica
 \newline\indent Università di Bologna \newline\indent
 Piazza di Porta San Donato 5, 40126 Bologna, Italy}
 \email{eugenio.vecchi2@unibo.it}



\thanks{The authors are members of INdAM. S. Biagi
is partially supported by the INdAM-GNAMPA project 
\emph{Metodi topologici per problemi al contorno associati a certe 
classi di equazioni alle derivate parziali}.
S. Dipierro and E. Valdinoci are members of AustMS.
S. Dipierro is supported by
the Australian Research Council DECRA DE180100957
``PDEs, free boundaries and applications''.
E. Valdinoci is supported by the Australian Laureate Fellowship
FL190100081
``Minimal surfaces, free boundaries and partial differential equations''.
E. Vecchi is partially supported
by the INdAM-GNAMPA project 
\emph{Convergenze variazionali per funzionali e operatori dipendenti da campi vettoriali}.}

 \date{\today}
 
 \maketitle
 
 \begin{center}
{\fmmfamily{ \Huge Dedicatoria. Al Ingenioso Hidalgo Don Ireneo.}}
 \end{center}\bigskip
 
\section{Introduction} 

In this paper we consider a nonlinear operator arising from the superposition of a classical $p$-Laplace operator
and a fractional $p$-Laplace operator, of the form
\begin{equation}\label{ATORE}\mathcal{L}_{p,s} = -\Delta_p+(-\Delta)^s_p\end{equation}
with~$s\in(0,1)$ and~$p\in [2,+\infty)$. Here the fractional $p$-Laplace operator is defined, up to a multiplicative
constant that we neglect, as
$$ (-\Delta)^s_p u(x):=2\int_{\R^n}\frac{|u(x)-u(y)|^{p-2} (u(x)-u(y))}{|x-y|^{n+ps}}\,dy
.$$

Given a bounded open set~$\Omega\subseteq\R^n$, we consider the eigenvalue problem for the operator~$\mathcal{L}_{p,s}$ with homogeneous
Dirichlet boundary conditions (i.e., the ei\-gen\-fun\-ctions are prescribed to vanish
in the complement of $\Omega$). In particular, we define~$\lambda_1(\Omega)$ to be the smallest of such eigenvalues
and~$\lambda_2(\Omega)$ to be the second smallest one (in the sense made precise in~\cite{BMV, GoelSreenadh}).

The main result that we present here is a version of the Hong--Krahn--Szeg\"{o} i\-ne\-qua\-li\-ty 
for the second Dirichlet eigenvalue~$\lambda_2(\Omega)$, according to the following statement:

 \begin{theorem} \label{main:thm}
  Let $\Omega\subseteq\R^n$ be a bounded open set. 
  Let $B$ be any Euclidean ball with volume $|\Omega|/2$. Then,
  \begin{equation} \label{eq:HKS}
   \lambda_2(\Omega) > \lambda_1(B).
  \end{equation}
  Furthermore, equality is never attained in 
  \eqref{eq:HKS}; however, the estimate is sharp in the following sense:
  if $\{x_j\}_j,\,\{y_j\}_j\subseteq\R^n$ are two sequences such that
  $$\lim_{j\to+\infty}|x_j-y_j| = +\infty,$$
  and if we define $\Omega_j := B_r(x_j)\cup B_r(y_j)$, then
  \begin{equation} \label{eq:optimal}
   \lim_{j\to+\infty}\lambda_2(\Omega_j) = \lambda_1(B_r).
  \end{equation}
 \end{theorem}

To the best of our knowledge, Theorem~\ref{main:thm} is new even in the linear case~$p=2$.
Also, an interesting  consequence of the fact that equality in~\eqref{eq:HKS} is never attained is that, for all~$c>0$,
the shape optimization problem
$$ \inf_{|\Omega|=c}\lambda_2(\Omega)$$
does not admit a solution. \medskip

Before diving into the technicalities of the proof of Theorem~\ref{main:thm},
we recall in the forthcoming Section~\ref{SEHI}
some classical motivations to study first and second eigenvalue problems,
then we devote Section~\ref{SEHI2} to showcase
the available results on the shape optimization problems related to the first and the second eigenvalues
of several elliptic operators.

\subsection{The importance of the first and second eigenvalues}\label{SEHI}

The notion of eigenvalue seems to date back to the 18th century, due to the works of Euler 
and Lagrange on rigid bodies. Possibly inspired by Helmholtz, in his study of integral operators~\cite{HILB}
Hilbert introduced the terminology of ``Eigenfunktion'' and ``Eigenwert''
from which the modern terminology of ``eigenfunction'' and ``ei\-gen\-va\-lue'' originated.

The analysis of eigenvalues also became topical in quantum mechanics,
being equivalent in this setting to the energy of a quantum state of a system,
and in general in the study of wave phenomena, to distinguish high and low frequency components.

In modern technologies, a deep understanding of eigenvalues has become a central theme of research,
especially due to the several ranking algorithms, such as PageRank (used by search engines as Google to rank the results)
and EigenTrust (used by peer-to-peer networks to establish a trust value
on the account of authentic and corrupted resources). In a nutshell, these algorithms typically have entries
(e.g. the page rank of a given website, or the trust value of a peer) that are measured as linear superpositions
of the other entries. For instance (see Section~2.1.1 in~\cite{BRIN}, neglecting for simplicity damping factors)
one can model the page rank~$p_i$ of website~$i$ in terms of
the ratio~$R_{ij}$ between the
number of links outbound from website~$j$ to page~$i$ and the total number of outbound links of website~$j$, namely
\begin{equation}\label{FIN9}
p_i=\sum_j R_{ij} p_j.\end{equation}
Whether this is a finite or infinite sum boils
down to a merely philosophical question, given the huge number of websites explored by Google, but let us stick for
the moment with the discrete case of finitely many websites.
Interestingly~$p_i$ basically counts the probability that a random surfer visits website~$i$ by following the available links in the web.

\begin{figure}
		\centering
		\includegraphics[width=.8\linewidth]{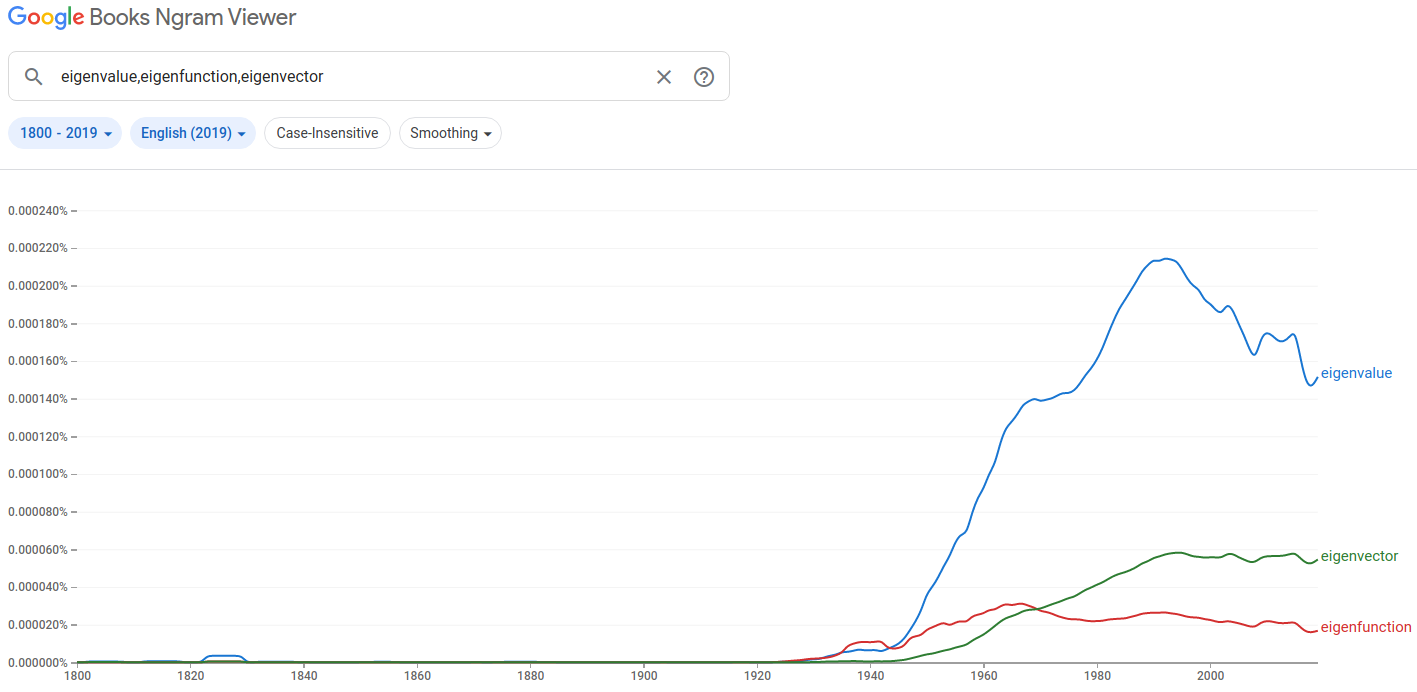}
\caption{}
	\label{PIPEDItangeFI}
\end{figure}

Now, in operator form, one can write~\eqref{FIN9}
as~$p=Rp$, with the matrix~$R$ known in principle from the outbound links of the websites and the ranking array~$p$ to be determined. Thus, up to diagonalizing~$R$, the determination of~$p$
reduces to the determination of the eigenvectors of~$R$, 
or equivalently to the determination of the eigenvectors of the inverse matrix~$A:=R^{-1}$,
and this task can be accomplished, for instance, by iterative algorithms.

The simplest of these algorithms used in PageRank is probably the power iteration method.
For instance, if one defines~$\eta_{k+1}:=\frac{A\eta_k}{|A\eta_k|}$, given
a random starting vector~$\eta_0$, it follows that~$\eta_k=\frac{A^k\eta_0}{|A^k\eta_0|}$
and consequently, if~$\eta_0=\sum_{j} c_j w_j$, being~$w_j$ the eigenvectors of~$A$ with corresponding eigenvalues~$\mu_1>\mu_2\ge\mu_3\ge...$ (normalized to have unit length), we find that
$$ \eta_k=\frac{\sum_{j} c_j \mu_j^k w_j}{\left|\sum_{j} c_j \mu_j^k w_j\right|}=
\frac{\sum_{j} d_{jk} w_j}{\left|\sum_{j} d_{jk} w_j\right|}=\frac{w_1+\sum_{j\ne1} d_{jk} w_j}{\left|{\rm{sign}} (c_1)w_1+\sum_{j\ne1} d_{jk} w_j\right|}.
$$
with~$d_{jk}:=\frac{c_j \mu_j^k}{c_1 \mu_1^k}$ (here we are assuming that the eigenvalues are positive
and that, in view of the randomness of~$\eta_0$, we have that~$c_1\ne0$).

Since
$$ \left|\sum_{j\ne1} d_{jk} w_j\right|^2=\sum_{j\ne1} d_{jk}^2=
\sum_{j\ne1}\frac{c_j^2 \mu_j^{2k}}{c_1^2 \mu_1^{2k}}=O\left(\frac{\mu_2^{2k}}{\mu_1^{2k}}\right),$$
it follows that
$$ \eta_k=w_1+O\left(\frac{\mu_2^{k}}{\mu_1^{k}}\right)
$$
and accordingly~$\eta_k$ approximates the eigenfunction~$w_1$ with a
convergence induced by the ratio~$\frac{\mu_2}{\mu_1}<1$.

That is, if~$\lambda_1<\lambda_2\le\lambda_3\le\dots$ are the eigenvalues of the matrix~$R$,
the above rate of convergence is dictated by the ratio~$\frac{\lambda_1}{\lambda_2}$
of the smallest and second smallest eigenvalues of~$R$. This is one simple, but, in our opinion quite convincing,
example of the importance of the first two eigenvalues in problems with concrete applications.

To confirm the importance of the notion of eigenvalues in the modern tech\-no\-lo\-gies,
see Figure~\ref{PIPEDItangeFI} for a Google Ngram Viewer charting the frequencies of use
of the words eigenvalue, eigenfunction and eigenvector in the last~120 years.

Also, to highlight the importance of the difference between the first and second eigenvalues,
see Figure~\ref{PIPEDItangeFI1}
for a Google Ngram Viewer charting the frequencies of use
of the words eigengap and spectral gap in the last~120 years.

In the case of the Google PageRank,
an efficient estimate of the eigengap taking into account the damping factor has been proposed in~\cite{TECH}.

\begin{figure}
		\centering
		\includegraphics[width=.8\linewidth]{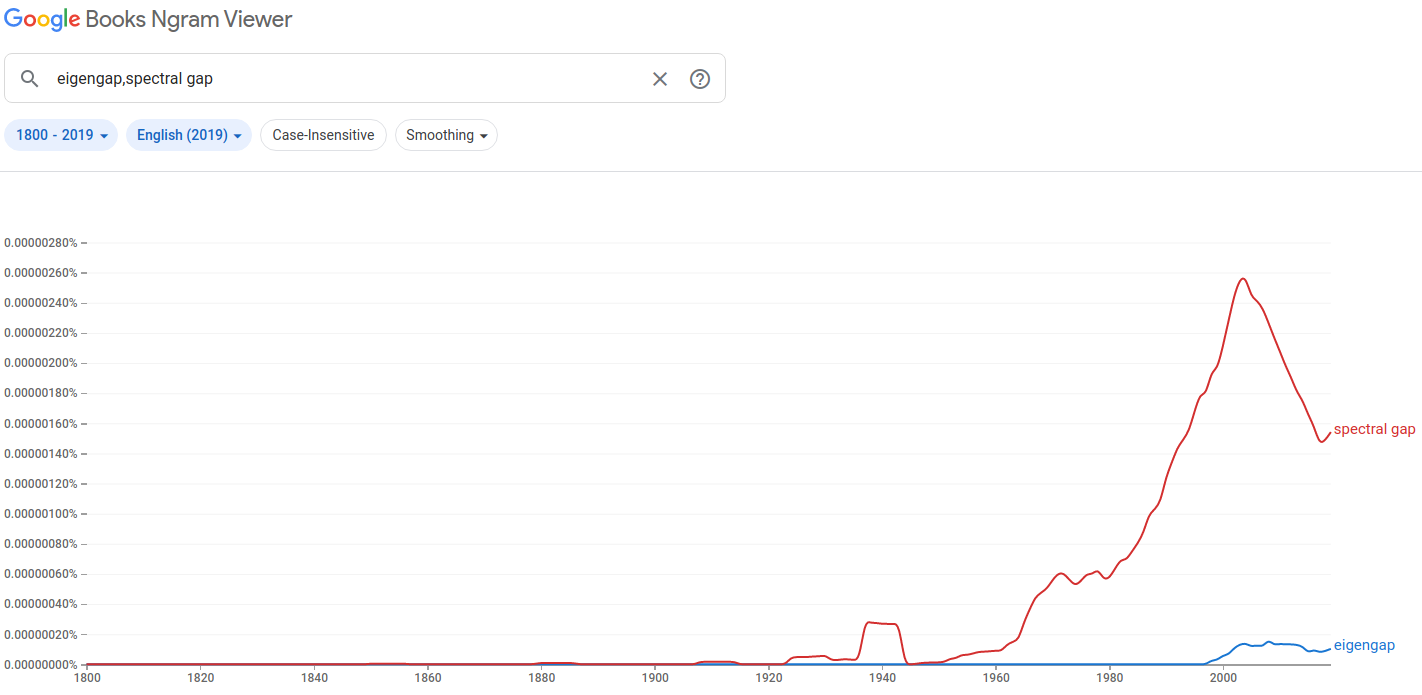}
\caption{}
	\label{PIPEDItangeFI1}
\end{figure}

\subsection{Shape optimization problems for the first and second eigenvalues in the context of elliptic (linear and nonlinear, classical and fractional) equations}\label{SEHI2}

Now we leave the realm of high-tech applications and we come back to the partial differential equations and fractional equations setting: in this fra\-me\-work, we recall here below some of the main results about
shape optimization problems related to the first and the second eigenvalues.

\subsubsection{The case of the Laplacian}\label{p2kd}

One of the classical shape optimization problem is related to the detection of
the domain that minimizes the first eigenvalue of the Laplacian with homogeneous boundary conditions.
This is the content of the Faber--Krahn inequality~\cite{FABER, KRA}, whose result can be stated
by saying that among all domains of fixed volume, the ball has the smallest first eigenvalue.

In particular, as a physical application, one has that that among all drums of equal area, the circular drum possesses the lowest voice, and this somewhat corresponds to our intuition, since a very elongated rectangular drum produces a high pitch
related to the oscillations along the short edge.

Another physical consequence of the Faber--Krahn inequality is that
among all the regions of a given volume with the boundary maintained at
a constant temperature, the one which dissipates heat at the
slowest possible rate is the sphere, and this also corresponds to our everyday life experience
of spheres mi\-ni\-mi\-zing 
contact with the external environment thus providing the optimal possible insulation.

{F}rom the mathematical point of view, the Faber--Krahn inequality also offers a classical stage for rearrangement
methods and variational characterizations of eigenvalues.

In view of the discussion in Section~\ref{SEHI}, the subsequent natural question 
in\-ve\-sti\-ga\-tes
the optimal shape of the second eigenvalue. This problem
is addressed by the Hong--Krahn--Szeg\"{o} inequality~\cite{7CK, 7CH, C7P},
which asserts that
among all domains of fixed volume, the disjoint union of two equal balls
has the smallest second eigenvalue.

Therefore, for the case of the Laplacian with homogeneous Dirichlet data, the shape optimization problems related to both the first and the second eigenvalues are solvable and the solution has a simple geometry.

It is also interesting to point out a conceptual connection between
the Faber--Krahn and the Hong--Krahn--Szeg\"{o} inequalities, in the sense that the proof of the second
typically uses the first one as a basic ingredient. More specifically, the strategy to
prove the Hong--Krahn--Szeg\"{o} inequality is usually:
\begin{itemize}
\item Use that in a connected open set all eigenfunctions except the \label{BULL}
first one must change sign,
\item Deduce that~$\lambda_2(\Omega)=\max\{\lambda_1(\Omega_+),\lambda_1(\Omega_-)\}$,
for suitable subdomain~$\Omega_+$ and~$\Omega_-$ which are either nodal domains
for the second eigenfunction, if~$\Omega$ is connected, or otherwise connected components of~$\Omega$,
\item Utilize the Faber--Krahn inequality to show that~$\lambda_1(\Omega_\pm)$ is reduced
if we replace~$\Omega_\pm$ with a ball of volume~$|\Omega_\pm|$,
\item Employ the homogeneity of the problem to deduce that the volumes of these two balls are equal.
\end{itemize}
That is, roughly speaking, a cunning use of the Faber--Krahn inequality allows one to reduce to the case
of disjoint balls, which can thus be addressed specifically.

\subsubsection{The case of the $p$-Laplacian}

A natural extension of the optimal shape results for the Laplacian recalled in Section~\ref{p2kd}
is the investigation of the non\-li\-near operator setting and in particular 
the case of the $p$-Laplacian.
This line of research was carried out in~\cite{FRANZ} in which a complete analogue of the results of
Section~\ref{p2kd} have been established for the $p$-Laplacian.
In particular, the first Dirichlet ei\-gen\-va\-lue of the $p$-Laplacian is minimized by the ball and
the second by any disjoint union of two equal balls.

We stress that, in spite of the similarity of the results obtained, the nonlinear case presents its
own specific peculiarities.
In particular, in the case of the $p$-Laplacian one can still define the first eigenvalue by minimization of
a Rayleigh quotient, in principle the notion of higher eigenvalues become more tricky,
since discreteness of the spectrum is not guaranteed and the eigenvalues theory for nonlinear operators
offers plenty of open problems at a fundamental level. 
For the second eingevalue however one can obtain a variational characterization in terms of a mountain-pass result,
still allowing the definition of a spectral gap between the smallest and the second smallest eigenvalue.

\subsubsection{The cases of the fractional Laplacian and of the fractional $p$-Laplacian}

We now consider the question posed by the minimization of the first and second eigenvalues in a nonlocal setting.

The optimal shape problems
for the first eigenvalue of the fractional Laplacian with homogeneous external datum was addressed in~\cite{BANL, IFBBPL, SIRE, CINTI},
showing that the ball is the optimizer.

As for the nonlinear case, the spectral properties of the fractional $p$-Laplacian possess their own 
special features, see~\cite{FRNPAL},
and they typically combine the dif\-fi\-cul\-ties coming from the nonlocal world with those arising from 
the theory of nonlinear operators.
In~\cite{IFBBPL} the optimal shape problem for the first Dirichlet eigenvalue of the fractional $p$-
Laplacian
was addressed, by detecting the optimality of the ball as a consequence of a general
P\`olya--Szeg\"o principle.

For the second eigenvalue, however, the situation in the
nonlocal case is quite different from the classical one,
since in general nonlocal energy functionals are deeply influenced by
the mutual position of the different connected
components of the domain, see~\cite{LIND}.

In particular, the counterpart of the Hong--Krahn--Szeg\"{o} inequality for the fractional Laplacian and the fractional $p$-Laplacian was established in~\cite{BP}
and it presents significant differences with the classical case: in particular, the shape optimizer for
the second eigenvalue of the fractional $p$-Laplacian with homogeneous external datum does not exist
and one can bound such an eigenvalue from below by the first eigenvalue of a ball with half of the volume
of the given domain (and this is the best lower bound possible, since the case of a domain consisting of two equal balls
drifting away from each other would attain such a bound in the limit).

\subsubsection{The case of mixed operators}

The study of mixed local/nonlocal operators has been recently received an increasing level of attention,
both in view of their intriguing mathematical structure, which combines the classical setting
and the features typical of nonlocal operators in a framework that is
not scale-invariant 
\cite{GaMe, JAKO1-05, JAKO2, BARLES08, Foondun, BISWAP10, CK, BAR12, CKSV, TESO17, DFR, SilvaSalort, GQR, BDVVAcc, ABA, BSM, dSOR, DPLV1, DPLV2, GarainKinnunen, GarainKinnunen2, GarainKinnunen3, GarainUkhlov, CABRE, BMV, BDVV, SalortVecchi},
and of their importance in practical applications such as the animal foraging hypothesis~\cite{DV, PAG}.

In regard to the shape optimization problem, a Faber--Krahn inequality for mixed local and nonlocal linear operators when~$p=2$ has been established in~\cite{BDVV2},
showing the optimality of the ball in the minimization of the first eigenvalue.
The corresponding inequality for the nonlinear setting presented in~\eqref{ATORE}
will be given here in the forthcoming Theorem~\ref{thm.FK}.

The inequality of Hong--Krahn--Szeg\"{o} type for mixed local and nonlocal linear operators 
presented in~\eqref{main:thm} would thus complete the study of the optimal shape problems for the first
and second eigenvalues of the operator in~\eqref{ATORE}.

\subsection{Plan of the paper}

The rest of this paper is organized as follows. Section~\ref{sec.Prel}
sets up the notation and collects some auxiliary results from the existing literature.

In Section~\ref{INTRERSEC} we discuss a regularity theory which, in our setting,
plays an important role in the proof of Theorem~\ref{main:thm} in allowing us to speak about nodal regions 
for the corresponding eigenfunction (recall the bullet point strategy pre\-sen\-ted 
on page~\pageref{BULL}).
In any case, this regularity theory holds in a more general setting and can well come in handy in other 
situations as well.

Finally, Section~\ref{LULT} introduces the corresponding Faber--Krahn inequality for 
the operator in~\eqref{ATORE}
and completes the proof of Theorem~\ref{main:thm}.

 \section{Preliminaries}\label{sec.Prel}
 
To deal with the nonlinear and mixed local/nonlocal operator in~\eqref{ATORE}, given
and open and bounded set~$\Omega\subseteq\R^n$, it is convenient to introduce the space 
 $$\mathcal{X}_0^{1,p}(\Omega)\subseteq W^{1,p}(\R^n),$$ 
 defined
 as the closure of $C_0^{\infty}(\Omega)$ with respect
 to the \emph{global norm}
 $$u\mapsto \Big(\int_{\R^n}|\nabla u|^p\,d x\Big)^{1/p}.$$
 We highlight that,
 since $\Omega$ is \emph{bounded}, $\mathcal{X}^{1,p}_0(\Omega)$ can be
 equivalently defined by taking the closure of $C_0^\infty(\Omega)$ with respect
 to the full norm
 $$u\mapsto \bigg(\int_{\R^n}|u|^p\,d x\bigg)^{1/p}
 +\bigg(\int_{\R^n}|\nabla u|^p\,dx\bigg)^{1/p};$$
 however, we stress that $\mathcal{X}_0^{1,p}(\Omega)$ \emph{is different}
 from the usual space $W^{1,p}_0(\Omega)$, which is 
 de\-fi\-ned
 as the closure of $C_0^{\infty}(\Omega)$ with respect to the norm
 $$u\mapsto \Big(\int_{\Omega}|\nabla u|^p\,d x\Big)^{1/p}.$$ 
 As a matter of fact, while
 the belonging of a function $u$ to $W^{1,p}_0(\Omega)$ only
 depends on its behavior \emph{inside of $\Omega$}
 (actually, $u$ does not even need to be defined outside of $\Omega$),
 the belonging of $u$ to $\mathcal{X}_0^{1,p}(\Omega)$ is a \emph{global}
 condition, and it depends on the behavior of $u$ \emph{on the whole space $\R^n$}
 (in particular, $u$ \emph{has to be defined} on $\R^n$).
 Just to give an example of the difference between these spaces,
 let $u\in C^\infty_0(\R^n)\setminus\{0\}$ be such that
 $$\mathrm{supp}(u)\cap \overline{\Omega} = \varnothing.$$
 Since $u\equiv 0$ inside of $\Omega$, we clearly have that
 $u\in W^{1,p}_0(\Omega)$; on the other hand, since $u\not\equiv 0$
 in $\R^n\setminus\Omega$,
 one has $u\notin \mathcal{X}_0^{1,p}(\Omega)$ (even if $u\in W^{1,p}(\R^n)$).
 \medskip
 
 Although they \emph{do not coincide}, the spaces
 $\mathcal{X}_0^{1,p}(\Omega)$ and $W^{1,p}_0(\Omega)$
 are re\-la\-ted: to be more precise, using \cite[Proposition\,9.18]{Brezis}
 and taking into account the definition
 of $\mathcal{X}_0^{1,p}(\Omega)$, one can see that
 \begin{itemize}
  \item[(i)] if $u\in W_0^{1,p}(\Omega)$, then 
  $u\cdot\mathbf{1}_\Omega\in \mathcal{X}_0^{1,p}(\Omega)$;
  \vspace*{0.1cm}
  
  \item[(ii)] if $u\in \mathcal{X}_0^{1,p}(\Omega)$, then
  $u\big|_{\Omega}\in W_0^{1,p}(\Omega)$.
 \end{itemize}
 Moreover, we can actually \emph{characterize} $\mathcal{X}_0^{1,p}(\Omega)$ as follows:
 $$\mathcal{X}_0^{1,p}(\Omega) = \{u\in W^{1,p}(\R^n):\,
 \text{$u\big|_\Omega\in W^{1,p}_0(\Omega)$ and $u = 0$ a.e.\,in $\R^n\setminus\Omega$}\}.$$
 The main issue in trying to use (i)-(ii) to identify
 $W_0^{1,p}(\Omega)$ with $\mathcal{X}_0^{1,p}(\Omega)$ is that,
 if $u$ is \emph{globally defined} and $u\in W^{1,p}(\R^n)$, then
 $$u\big|_{\Omega}\in W^{1,p}_0(\Omega)\,\,\Rightarrow\,\,u\cdot \mathbf{1}_\Omega\in 
 \mathcal{X}_0^{1,p}(\Omega);$$
 however, we cannot say (in general) that
 $u\neq u\cdot\mathbf{1}_\Omega$. 
 Even if they cannot allow to identify
 $\mathcal{X}_0^{1,p}(\Omega)$ with $W_0^{1,p}(\Omega)$, 
 assertions (i)-(ii) can be used to deduce several
 properties of the space
 $\mathcal{X}_0^{1,p}(\Omega)$ starting from their
 analog in $W_0^{1,p}(\Omega)$; for example, we have
 the following fact, which shall be used in the what follows:
 $$u\in \mathcal{X}_0^{1,p}(\Omega)
 \,\,\Rightarrow\,\,|u|,\,u^+,\,u^-\in \mathcal{X}_0^{1,p}(\Omega).$$
 \begin{remark} \label{rem:Omegaregular}
  In the particular case when the open set $\Omega$ is of class $C^1$, 
  it follows from \cite[Proposition\,9.18]{Brezis} that, if $u\in W^{1,p}(\R^n)$
  and $u = 0$ a.e.\,in $\R^n\setminus\Omega$, then
  $$u\big|_{\Omega}\in W_0^{1,p}(\Omega).$$
  As a consequence, we have
  $$\mathcal{X}_0^{1,p}(\Omega) =
  \{u\in W^{1,p}(\Omega):\,\text{$u = 0$ a.e.\,in $\R^n\setminus\Omega$}\}.$$
  This fact shows that, when $\Omega$ is sufficiently regular,
  $\mathcal{X}_0^{1,p}(\Omega)$ \emph{coincides}
  with the space $\mathbb{X}_p(\Omega)$ introduced in
  \cite{BDVV} (for $p = 2$) and in \cite{BMV} (for a general $p > 1$).
 \end{remark}
 For future reference, we introduce the following set
 \begin{equation}  \label{eq:defM}
  \mathcal{M}(\Omega) := \bigg\{u\in\mathcal{X}_0^{1,p}(\Omega):\,
 \int_{\R^n}|u|^p\,d x = 1\bigg\}.
 \end{equation}
 
 After these preliminaries, we can turn our attention
 to the \emph{Dirichlet pro\-blem} for the operator
 $\LL_{p,s}$. 
 Throughout the rest of this paper, to simplify the notation we set
 \begin{equation} \label{eq:defJp}
  J_p(t) := |t|^{p-2}t\qquad {\mbox{ for all }}t\in\R.
 \end{equation}
 Moreover, we define
 $$p^* := \begin{cases}
 \dfrac{np}{n-p} & \text{if $p < n$}, \\
 +\infty & \text{if $p\geq n$},
 \end{cases}\quad\text{and}\quad
 (p^*)' := \begin{cases}
 \dfrac{p^*}{p^*-1} & \text{if $p < n$}, \\
 1 & \text{if $p\geq n$}.
 \end{cases}
 $$
 \begin{definition} \label{def:weaksol}
  Let $q\geq (p^*)'$, and let
  $f\in L^q(\Omega)$. We say that a function~$u\in W^{1,p}(\R^n)$ is a 
  \emph{weak solution}
  to the equation
  \begin{equation} \label{eq:mainPDE}
  \LL_{p,s}u = f \qquad\text{in $\Omega$}
 \end{equation}
 if, for every
  $\phi\in \mathcal{X}_0^{1,p}(\Omega)$, the following identity is satisfied
  \begin{equation} \label{eq:defweaksol}
  \begin{split}
   & \int_{\Omega}|\nabla u|^{p-2}\langle\nabla u, \nabla \phi\rangle\,d x
   \\
   & \qquad\qquad
   + \iint_{\R^{2n}}\frac{J_p(u(x)-u(y))(\phi(x)-\phi(y))}{|x-y|^{n+ps}}\,dx\,dy
   = \int_{\Omega}f\phi\,d x,
   \end{split}
  \end{equation}
  Moreover, given any $g\in W^{1,p}(\R^n)$, we say that a function
  $u\in W^{1,p}(\R^n)$ is a weak solution to the 
  \emph{$(\LL_{p,s})$-Dirichlet
  problem}
  \begin{equation} \label{eq:DirPbfg}
   \begin{cases}
   \LL_{p,s}u = f & \text{in $\Omega$}, \\
   u = g & \text{in $\R^n\setminus\Omega$}, 
   \end{cases}
  \end{equation}
  if $u$ is a weak solution to \eqref{eq:mainPDE} and, in addition,
  $$u- g\in\mathcal{X}_0^{1,p}(\Omega).$$
  \end{definition}
  \begin{remark} \label{rem:welldef}
  (1)\,\,We point out that the above definition is well-posed: indeed, if~$u,v\in W^{1,p}(\Omega)$, by H\"older's inequality
  and \cite[Proposition\,2.2]{DRV} we get
  \begin{align*}
   & \iint_{\R^{2n}}\frac{|u(x)-u(y)|^{p-1}|v(x)-v(y)|}{|x-y|^{n+ps}}\,dx\,d y
   \ \\
   & \qquad 
   \leq
   \bigg(\iint_{\R^{2n}}\frac{|u(x)-u(y)|^{p}}{|x-y|^{n+ps}}\,dx\,d y
   \bigg)^{1/p}
   \bigg(\iint_{\R^{2n}}\frac{|v(x)-v(y)|^{p}}{|x-y|^{n+ps}}\,dx\,d y
   \bigg)^{1/p} \\[0.2cm]
   & \qquad
   \leq \mathbf{c}\,\|u\|_{W^{1,p}(\R^n)}\,\|v\|_{W^{1,p}(\R^n)} < +\infty.
  \end{align*}
  Moreover, since $f\in L^q(\Omega)$ and $q\geq (p^*)'$, again by H\"older's
  inequality and by the Sobolev Embedding Theorem (applied
  here to $v\in W^{1,p}(\R^n)$), we have
  \begin{align*}
   \int_{\Omega}|f||v|\, dx \leq \|f\|_{L^{(p^*)'}(\Omega)}\,\|v\|_{L^{p^*}(\Omega)}
   <+\infty.
  \end{align*}
  
  (2)\,\,If $W^{1,p}(\R^n)$ is 
  is a weak solution to the 
  {$(\LL_{p,s})$-Dirichlet
  problem} \eqref{eq:DirPbfg}, it follows from the definition
  of $\mathcal{X}_0^{1,p}(\Omega)$ that
  $$(u-g)\big|_{\Omega}\in W^{1,p}_0(\Omega)\qquad\text{and}\qquad
  \text{$u = g$ a.e.\,in $\R^n\setminus\Omega$}.$$
  Thus, $\mathcal{X}_0^{1,p}(\Omega)$ is the `right space' for the weak formulation
  of \eqref{eq:DirPbfg}.
 \end{remark}
 With Definition \ref{def:weaksol} at hand, we now introduce the notion
 of Dirichlet ei\-gen\-va\-lue/ei\-gen\-fun\-ction
 for the operator $\LL_{p,s}$.
 \begin{definition} \label{def:eigenval} 
 We say that $\lambda\in\R$ is a \emph{Dirichlet eigenvalue}
 for $\LL_{p,s}$ if there exists a solution
 $u\in W^{1,p}(\Omega)\setminus\{0\}$ of the
 $(\LL_{p,s})$-Dirichlet problem
 \begin{equation} \label{eq:DirPbEigen}
  \begin{cases}
  \LL_{p,s}u = \lambda|u|^{p-2}u & \text{in $\Omega$}, \\
  u = 0 & \text{in $\R^n\setminus\Omega$}.
  \end{cases}
 \end{equation}
 In this case, we say that $u$ is an \emph{eigenfunction} associated with $\lambda$.
 \end{definition}
 \begin{remark} \label{rem:defwellposedEigenval}
  We point out that Definition 
  \ref{def:eigenval} is {well-posed}. Indeed,
  if $u$ is any function in $W^{1,p}(\R^n)$, 
  by the Sobolev Embedding Theorem we have
  $$f := |u|^{p-2}u\in L^{\frac{p^*}{p-1}}(\Omega);$$
  then, a direct computation shows that $q := p^*/(p-1)\geq (p^*)'$. As a consequence,
  the notion of weak solution 
  for \eqref{eq:DirPbEigen} agrees with the one contained in Definition~\ref{def:weaksol}. 
  In particular, if $u$ is an 
  {eigenfunction} associated with some eigenvalue $\lambda$, then
  $$u\in\mathcal{X}_0^{1,p}(\Omega),$$
  and thus $u\big|_\Omega\in W_0^{1,p}(\Omega)$ and $u = 0$ a.e.\,in $\R^n\setminus\Omega$.  
 \end{remark}
 After these definitions, we close 
 the section by reviewing some results
 about eigenvalues/eigenfucntions for $\LL_{p,s}$
 which shall be used here below.
 \medskip
 
 To begin with, we recall the following result proved in \cite{BMV} which establishes the existence
 of the smallest eigenvalue and detects its basic properties.
 \begin{proposition}[\protect{\cite[Proposition\,5.1]{BMV}}] \label{prop.BMV}
 The smallest ei\-gen\-va\-lue 
 $\lambda_1(\Omega) $ for the operator $\LL_{p,s}$ is strictly positive and satisfies the following properties: 
 \begin{enumerate}
  \item $\lambda_1(\Omega)$ is simple;
  \item the eigenfunctions associated with $\lambda_1(\Omega)$ do not change sign in $\R^n$;
  \item every eigenfunction associated to an eigenvalue 
 $$\lambda > \lambda_1(\Omega)$$ 
 is nodal, i.e., sign changing. 
 \end{enumerate}
 Moreover, $\lambda_1(\Omega)$ admits the following variational characterization
 \begin{equation} \label{eq:deflambdavar}
  \lambda_1(\Omega) = \min_{u\in\mathcal{M}(\Omega)}
 \bigg(\int_\Omega|\nabla u|^p\,dx + \iint_{\R^{2n}}
 \frac{|u(x)-u(y)|^p}{|x-y|^{n+ps}}\,dx\,dy\bigg),
 \end{equation}
 where $\mathcal{M}(\Omega)$ is as in \eqref{eq:defM}. The minimum is always attained, and the
 ei\-gen\-fun\-ctions for $\LL_{p,s}$ associated with $\lambda_1(\Omega)$
 are precisely the minimizers in \eqref{eq:deflambdavar}.
 \end{proposition}
 We observe that, on account of Proposition \ref{prop.BMV},
 there exists a \emph{unique non-negative}
 eigenfunction $u_0\in\mathcal{M}(\Omega)\subseteq\mathcal{X}_0^{1,p}(\Omega)$ 
 associated with $\lambda_1(\Omega)$;
 in par\-ti\-cu\-lar, $u_0$ is a minimizer in \eqref{eq:deflambdavar}, so that
 \begin{equation} \label{eq:uzerolambda1}
  \lambda_1(\Omega) =  
  \int_\Omega|\nabla u_0|^p\,dx + \iint_{\R^{2n}}
 \frac{|u_0(x)-u_0(y)|^p}{|x-y|^{n+ps}}\,dx\,dy.
 \end{equation}
 We shall refer to $u_0$ as the \emph{principal eigenfunction} of $\LL_{p,s}$.
 \medskip
 
 The next result was proved in \cite{GoelSreenadh} and concerns the \emph{second}
 eigenvalue for $\LL_{p,s}$.
 \begin{theorem}[\protect{\cite[Section\,5]{GoelSreenadh}}]\label{thm:GS}
 We define:
 \begin{equation} \label{eq:deflambda2}
  \lambda_2(\Omega) := 
  \inf_{f\in \mathcal{K}}\max_{u\in \mathrm{Im}(f)}
  \bigg\{\int_{\Omega}|\nabla u|^p\,d x
  + \iint_{\R^{2n}}\frac{|u(x)-u(y)|^p}{|x-y|^{n+ps}}\,dx\,dy\bigg\}, 
 \end{equation}
 where $\mathcal{K} := \{f:S^1\to\mathcal{M}(\Omega):\,\text{$f$ is continuous and odd}\}$, with~$\mathcal{M}(\Omega)$ as in~\eqref{eq:defM}.
 
  Then:
 \begin{enumerate}
 \item $\lambda_2(\Omega)$ is an eigenvalue for $\LL_{p,s}$;
 \item $\lambda_2 (\Omega) > \lambda_1(\Omega)$;
 \item If $\lambda > \lambda_1(\Omega)$ is an 
 eigenvalue for $\LL_{p,s}$, then $\lambda \geq \lambda_2(\Omega)$.
 \end{enumerate}
 \end{theorem}
In the rest of this paper, we shall refer to $\lambda_1(\Omega)$ and $\lambda_2(\Omega)$
 as, respectively, the \emph{first and second eigenvalue}
 of $\LL_{p,s}$ (in $\Omega$). We notice that,
 as a con\-se\-que\-nce of~\eqref{eq:deflambdavar}-\eqref{eq:deflambda2}, 
 both $\lambda_1(\cdot)$ and $\lambda_2(\cdot)$ are \emph{translation-invariant},
 that is,
 $$\lambda_1(x_0+\Omega) = \lambda_1(\Omega)\qquad{\mbox{ and }} \qquad
 \lambda_2(x_0+\Omega) = \lambda_2(\Omega).$$
 To proceed further, we now recall the following \emph{global boundedness}
 result for the ei\-gen\-fun\-ctions of
 $\LL_{p,s}$ (associated with \emph{any}
 eigenvalue $\lambda$) established in \cite{BMV}.
 \begin{theorem}[\protect{\cite[Theorem\,4.4]{BMV}}]\label{thm:GlobalBd}
  Let $u\in 
  \mathcal{X}_0^{1,p}(\Omega)\setminus\{0\}$ be an eigenfunction for~$\LL_{p,s}$, as\-so\-cia\-ted
  with an eigenfunction $\lambda \geq \lambda_1(\Omega)$. Then, 
  $u\in L^\infty(\R^n)$.
 \end{theorem}
 \begin{remark} \label{rem:proofBd}
  Actually, in \cite[Theorem\,4.4]{BMV}
 it is proved the global
 bound\-ed\-ness of any \emph{non-negative} weak solution to the more general Dirichlet problem
 $$\begin{cases}
  \LL_{p,s} = f(x,u) & \text{in $\Omega$}, \\
  u \equiv 0 & \text{a.e.\,in $\R^n\setminus\Omega$},
  \end{cases}$$
  where $f:\Omega\times\R\to\R$ is a Carath\'{e}odory function satisfying the properties
  \begin{itemize}
   \item[(a)] $f(\cdot,t)\in L^\infty(\Omega)$ for every $t\geq 0$;
   \item[(b)]  There exists a constant $c_p > 0$ such that
   $$|f(x, t)| \leq c_p(1+t^{p-1})\qquad\text{for a.e.\,$x\in\Omega$ and every $t\geq 0$}.$$
  \end{itemize}
  However, by scrutinizing the proof of the theorem, it is easy
  to check that the same argument
  can be applied to our context, where we have
  $$f(x,t) = \lambda|t|^{p-2}t\qquad
  {\mbox{ for all }} x\in\Omega {\mbox{ and }} t\in\R,$$
  but we do not make any assumption on the sign of $u$ 
  (see also \cite[Proposition\,4]{SerVal2}).
 \end{remark}
 Finally, we state here an algebraic lemma which shall be useful in the 
 forth\-co\-ming computations.
 \begin{lemma}\label{lem:algebrico}
 Let $1<p<+\infty$ be fixed. Then, the following facts hold.
 \begin{enumerate}
  \item For every $a,b\in \mathbb{R}$ 
 such that $ab\leq 0$, it holds that
 \begin{equation*}
 J_p(a-b)a \geq 
 \begin{cases}
 |a|^p - (p-1)|a-b|^{p-2}ab, & \text{if $1<p\leq 2$}, \\[0.1cm]
 |a|^p - (p-1)|a|^{p-2}ab, & \text{if $p> 2$}.
 \end{cases}
 \end{equation*}
 \item There exists a constant $c_p > 0$ such that
 $$|a-b|^p \leq |a|^p+|b|^p + c_p\big(|a|^2+|b|^2\big)^{\frac{p-2}{2}}|ab|,\qquad
 \forall\,\,a,b\in\R.$$
 \end{enumerate}
 \end{lemma}

 \section{Interior regularity of the eigenfunctions} \label{INTRERSEC}
 
 In this section we prove the \emph{interior H\"older regularity}
 of the eigenfunctions for~$\LL_{p,s}$, which is
 a fundamental ingredient for the proof of
 Theorem \ref{main:thm}.
 As a matter of fact, on account of Theorem \ref{thm:GlobalBd}, we establish
 the interior H\"older regularity for any \emph{bounded} weak solution
 of the non-homogeneous equation
 \eqref{eq:mainPDE}, when 
 $$f\in L^\infty(\Omega).$$
  In what follows, we tacitly understand that 
  $$\text{$2\leq p \leq n$ and $s\in (0,1)$};$$
  mo\-re\-o\-ver, $\Omega\subseteq\R^n$ is a bounded open set and
  $f\in L^\infty(\Omega)$.
  \begin{remark} \label{rem:choicepleqn}
   The reason why we restrict ourselves to consider
   $2\leq p\leq n$ follows from the definition
   of  weak solution to \eqref{eq:mainPDE}. 
   
   Indeed, if $u$ is a weak solution
   to \eqref{eq:mainPDE}, then by definition we have $u\in W^{1,p}(\R^n)$;
   as a consequence, if $p > n$, by the classical Sobolev Embedding Theorem
   we can immediately conclude that $u\in C^{0,\gamma}(\R^n)$, where
   $\gamma = 1-n/p$.
  \end{remark}
  In order to state (and prove) the main result
  of this section, we need to 
  fix a notation:
  for every $z\in\R^n,\,\rho > 0$ and $u\in L^p(\R^n)$, we define
  $$\mathrm{Tail}(u,z,\rho) := \bigg(\rho^p\int_{\R^n\setminus
  B_\rho(z)}\frac{|u|^{p}}{|x-z|^{n+ps}}\,dx\bigg)^{1/p}.$$
  The quantity $\mathrm{Tail}(u,z,\rho)$ 
  is referred to as the \emph{$(\LL_{p,s})$-tail} of $u$, see e.g.~\cite{KMS, AGNE}.
  
  \begin{theorem} \label{thm:mainHolder}
    Let $f\in L^\infty(\Omega)$, and let $u\in W^{1,p}(\R^n)\cap L^\infty(\R^n)$ be
    a weak so\-lu\-tion to \eqref{eq:mainPDE}. Then, there exists
    some $\beta = \beta(n,s,p)\in (0,1)$ such that
    $u\in C^{0,\beta}_{\loc}(\Omega)$.
    \vspace*{0.1cm}
    
    More precisely, for every ball $B_{R_0}(z)\Subset\Omega$ we have the estimate
    \begin{equation} \label{eq:HolderestimMain}
     [u]_{C^{0,\beta}(B_{R_0}(z))}^p
     \leq C\Big(\|f\|_{L^\infty(\Omega)}
     + \|u\|_{L^\infty(\Omega)}^p+
     \mathrm{Tail}(u,z,R_1)^p+1\Big),
    \end{equation}
	where 
	$$R_1 := R_0 + \frac{\mathrm{dist}(B_{R_0}(z),\de\Omega)}{2}$$
	and $C > 0$ is a constant independent of $u$ and~$R_1$.
  \end{theorem}
  
  In order to prove Theorem \ref{thm:mainHolder}, we follow the
  approach in \cite{BLS}; broadly put, the main idea behind this approach
  is to \emph{transfer} to the solution $u$ the oscillation estimates proved
  in \cite{GarainKinnunen} for the \emph{$\LL_{p,s}$-harmonic functions}.
  \medskip
  
  To begin with, we establish the following basic existence/uniqueness
  result for the weak solutions to the $(\LL_{p,s})$-Dirichlet problem
  \eqref{eq:DirPbfg}.
\begin{proposition} \label{prop:existuniqweaksol}
  Let $f\in L^\infty(\Omega)$ and $g\in W^{1,p}(\R^n)$ be fixed. Then,
  there exists a unique solution $u = u_{f,\,g}\in W^{1,p}(\R^n)$
  to the Dirichlet problem \eqref{eq:DirPbfg}.
  \end{proposition}
  \begin{proof}
   We consider the space 
   $$\mathbb{W}(g) := \{u\in W^{1,p}(\R^n):\,u-g\in\mathcal{X}_0^{1,p}(\Omega)\},$$
   and the functional $J:\mathbb{W}(g)\to\R$ defined as follows:
   \begin{align*}
    J(u) & := \frac{1}{p}\int_{\Omega}
   |\nabla u|^{p}\,d x
   + \frac{1}{p}\iint_{\Omega\times\Omega}\frac{|u(x)-u(y)|^p}{|x-y|^{n+ps}} \\[0.1cm]
   & \qquad\quad
   + \frac{2}{p}\iint_{\Omega\times(\R^n\setminus\Omega)}\frac{|u(x)-g(y)|^p}{|x-y|^{n+ps}}
   -\int_{\Omega}fu\,dx.
   \end{align*}
   On account of \cite[Remark~2.13]{BLS}, we have that
   $J$ is \emph{strictly convex}; hence, by using the
   Direct Methods in the Calculus of Variations, we derive that
   $J$ has a unique minimizer $u = u_{f,\,g}$ on $\mathbb{W}(g)$, which
   is the unique weak solution to \eqref{eq:DirPbfg}.
  \end{proof}
  Thanks to Proposition \ref{prop:existuniqweaksol}, we can
  prove the following result:
  
  \begin{lemma} \label{lem:trequattroBrasco}
   Let $f\in L^\infty(\Omega)$ and let $u\in W^{1,p}(\R^n)$ be a weak solution
   to \eqref{eq:mainPDE}. Moreover, let $B$ be a given
   Euclidean ball such that $B\Subset\Omega$, and let
   $v\in W^{1,p}(\R^n)$ be the unique weak solution
   to the Dirichlet problem
   \begin{equation} \label{eq:DirvLemma}
    \begin{cases}
     \LL_{p,s}v = 0 & \text{in $\Omega$}, \\
     v = u & \text{in $\R^n\setminus\Omega$}.
    \end{cases}
   \end{equation}
   Then, there exists a constant $C = C(n,s,p) > 0$ such that
   \begin{equation} \label{eq:estimLemmatrequattroGagl}
   [u-v]_{W^{s,p}(\R^n)}^p \leq C|B|^{p'-\frac{p'(n-sp)}{np}}\|f\|_{L^\infty(\Omega)}^{p'}.
   \end{equation}
   In particular, we have
   \begin{equation} \label{eq:estimLemmatrequattroMean}
    \meanint_B|u-v|^p\,dx 
    \leq C|B|^{p'-\frac{p'(n-sp)}{np}+\frac{sp}{n}-1}\|f\|_{L^\infty(\Omega)}^{p'}.
   \end{equation}
  \end{lemma}
  \begin{proof}
   We observe that the existence of $v$
   is ensured by
   Proposition \ref{prop:existuniqweaksol}. 
   Then, ta\-king into account that $u$ is a weak solution
   to \eqref{eq:mainPDE} and $v$ is the weak solution to \eqref{eq:DirvLemma},
   for every $\phi\in \mathcal{X}^{1,p}_0(B)$ we get
   \begin{align*}
    & \int_{B}\big(|\nabla u|^{p-2}\langle \nabla u,\nabla\phi\rangle-
    |\nabla v|^{p-2}\langle \nabla v,\nabla\phi\rangle\big)dx  \\
    & \quad
    + \iint_{\R^{2n}}\frac{\big(J_p(u(x)-u(y))-J_p(v(x)-v(y))\big)(\phi(x)-\phi(y))}{|x-y|^{n+ps}}
    \,dx\,dy
    = \int_B f\phi. 
   \end{align*}
  Choosing, in particular, $\phi := u-v$ (notice that, since 
  $v$ is a weak solution of \eqref{eq:DirvLemma}, by definition we have
  $v-u\in\mathcal{X}_0^{1,p}(\Omega)$), we obtain
  \begin{equation} \label{eq:dastimareLemmatrequattro}
   \begin{split}
   & \int_{\Omega}\mathcal{B}(\nabla u,\nabla v)\,dx
   + \iint_{\R^{2n}}
   \frac{\big(J_p(t_1)-J_p(t_2)\big)(t_1-t_2)}{|x-y|^{n+ps}}
    \,dx\,dy \\
    & \qquad
    = \int_B f(u-v)\,d x,
   \end{split}
  \end{equation}
  where 
  $t_1:= u(x)-u(y),\,t_2 := v(x)-v(y)$ and
  $$\mathcal{B}(a,b) := |a|^p+|b|^p-(|a|^{p-2}+|b|^{p-2})\langle a,b\rangle
  \qquad {\mbox{ for all }} a,b\in\R.$$
  Now, an elementary computation based on Cauchy-Schwarz's inequality gives
  \begin{equation} \label{eq:estimpartelocale}
   \mathcal{B}(a,b)\geq 0\qquad{\mbox{ for all }} a,b\in\R.
  \end{equation}
Moreover, since $p\geq 2$, by exploiting \cite[Remark~A.4]{BLS} we have
  \begin{equation} \label{eq:estimpartenonlocale}
   \big(J_p(t_1)-J_p(t_2)\big)(t_1-t_2)\geq \frac{1}{C}|t_1-t_2|^p,
  \end{equation}
  where $C = C(p) > 0$ is a constant only depending on $p$.
  Thus, by combining \eqref{eq:dastimareLemmatrequattro}, 
  \eqref{eq:estimpartelocale} and~\eqref{eq:estimpartenonlocale}, we obtain
  the following estimate:
  \begin{align*}
   & [u-v]_{W^{s,p}(\R^n)}^p
   = \iint_{\R^{2n}}\frac{|t_1-t_2|^p}{|x-y|^{n+ps}}\,dx\,dy
   \\
   & \qquad \leq C\bigg(\int_{\Omega}\mathcal{B}(\nabla u,\nabla v)\,dx
   + \iint_{\R^{2n}}
   \frac{\big(J_p(t_1)-J_p(t_2)\big)(t_1-t_2)}{|x-y|^{n+ps}}
    \,dx\,dy\bigg) \\[0.1cm]
    & \qquad \leq C\int_B f(u-v)\,dx \\&\qquad
    \leq C\|f\|_{L^\infty(\Omega)}\int_B|u-v|\,dx
    \\
    & \qquad
    \leq C\,|B|^{1-\frac{1}{p^*_s}}\|f\|_{L^\infty(\Omega)}
    \|u-v\|_{L^{p^*_s}(B)},
  \end{align*}
  where 
 we have also used the H\"older's inequality and~$p^*_s > 1$ is the so-called fractional critical exponent, that is, 
  $$p^*_s := \frac{np}{n-sp}.$$
  Finally, by applying the fractional Sobolev
  inequality to $\phi = u-v$ (notice that $\phi$ is compactly supported in
  $B$), we get
  $$[u-v]_{W^{s,p}(\R^n)}^p
  \leq C\,|B|^{1-\frac{1}{p^*_s}}\|f\|_{L^\infty(\Omega)}
   [u-v]_{W^{s,p}(\R^n)},$$
  and this readily yields the desired \eqref{eq:estimLemmatrequattroGagl}.
  To prove \eqref{eq:estimLemmatrequattroMean} we observe that,
  by using the H\"older inequality and again the fractional Sobolev inequality, we have
  \begin{align*}
   \meanint_B|u-v|^p\,dx 
   & \leq \bigg(\meanint_B|u-v|^{p^*_s}\,dx\bigg)^{\frac{p}{p^*_s}}
   \leq C\,|B|^{-\frac{p^*_s}{p}}\,[u-v]_{W^{s,p}(\R^n)}^{p};
  \end{align*}
  thus, estimate
  \eqref{eq:estimLemmatrequattroMean} follows directly from \eqref{eq:estimLemmatrequattroGagl}. 
  \end{proof}
  Using Lemma \ref{lem:trequattroBrasco}, we can prove the following
  \emph{excess decay estimate}.
  \begin{lemma} \label{prop:decayestimate}
   Let $f\in L^\infty(\Omega)$ and let $u\in W^{1,p}(\R^n)$ be a weak solution
   to \eqref{eq:mainPDE}. 
   Moreover, let $x_0\in\Omega$ and let 
   $R \in (0,1)$ be such that $B_{4R}(x_0)\Subset
   \Omega$.

   Then, for every $0<r\leq R$ we have the estimate
   \begin{equation} \label{eq:excessdecay}
   \begin{split}
    & \meanint_{B_r(x_0)}
    |u-\overline{u}_{x_0,r}|^p\,dx 
    \leq C\bigg(\frac{R}{r}\bigg)^n\,R^{\gamma}\,\|f\|_{L^\infty(\Omega)}^{p'}
    \\
    & \qquad
     +C\bigg(\frac{r}{R}\bigg)^{\alpha p}
     \bigg(R^\gamma\,\|f\|_{L^\infty(\Omega)}^{p'}
     + \meanint_{B_{4R}(x_0)}|u|^p\,dx 
     + \mathrm{Tail}(u,x_0,4R)^p\bigg),
   \end{split}
   \end{equation}
   where $C,\,\gamma$ and $\alpha$ are positive constants only depending
   on $n$, $s$ and~$p$.
  \end{lemma}
  
  \begin{proof}
   Let $v\in W^{1,p}(\R^n)$ be the unique weak solution to the problem
   \begin{equation} \label{eq:pbSolvedudecay}
    \begin{cases}
   \LL_{p,s}v = 0 & \text{in $B_{3R}(x_0)$}, \\
   v = u & \text{on $\R^n\setminus B_{3R}(x_0)$}.
   \end{cases}
   \end{equation}
We stress that the existence of $v$ is guaranteed by Proposition \ref{prop:existuniqweaksol}.
   We also observe that, for e\-ve\-ry $r\in(0, R]$, we have
   that
   \begin{equation*}
    |\overline{u}_{x_0,r}-\overline{v}_{x_0,r}|^p
    = \bigg|\meanint_{B_r(x_0)}(u-v)\,dx\bigg|^p\leq 
    \meanint_{B_r(x_0)}|u-v|^p\,dx.
   \end{equation*}
 As a consequence, we obtain
   \begin{equation} \label{eq:toestimI1I2}
   \begin{split}
    \meanint_{B_r(x_0)}
    |u-\overline{u}_{x_0,r}|^p\,dx 
    & \leq \kappa\meanint_{B_r(x_0)}
    |u-v|^p\,dx 
    + \kappa\meanint_{B_r(x_0)}
    |v-\overline{v}_{x_0,r}|^p\,dx 
    \\
    & \qquad
    + \kappa\meanint_{B_r(x_0)}
    |\overline{u}_{x_0,r}-\overline{v}_{x_0,r}|^p\,dx \\[0.2cm]
    & \leq \kappa\bigg(\meanint_{B_r(x_0)}
    |u-v|^p\,dx+\meanint_{B_r(x_0)}
    |v-\overline{v}_{x_0,r}|^p\,dx\bigg),
   \end{split}
   \end{equation}
   where $\kappa = \kappa_p > 0$ is a constant only depending on $p$.
   
   Now, since $B_{3R}(x_0)\Subset
   \Omega$ and $v$ is the weak solution
   to \eqref{eq:pbSolvedudecay}, by Lemma \ref{lem:trequattroBrasco}
   we have
   \begin{equation} \label{eq:estimfI}
   \begin{split}
    \meanint_{B_r(x_0)}
    |u-v|^p\,dx & 
    \leq C\,r^{np'-\frac{p'(n-sp)}{p}+sp-n}\|f\|_{L^\infty(\Omega)}^{p'}
    \\
    & \leq
    C\,\bigg(\frac{R}{r}\bigg)^n\,R^{np'-\frac{p'(n-sp)}{p}+sp-n}
    \|f\|_{L^\infty(\Omega)}^{p'}.
    \end{split}
    \end{equation}
   On the other hand, since $v\in W^{1,p}(\R^n)$ and $v$ is
   $\LL_{p,s}$-harmonic in $B_{3R}(x_0)$ (that is, $\LL_{p,s}v = 0$ in the weak sense),
   we can apply \cite[Theorem\,5.1]{GarainKinnunen}, obtaining
   \begin{equation} \label{eq:estimsI}
    \begin{split}
     \meanint_{B_r(x_0)}
    |v-\overline{v}_{x_0,r}|^p\,dx    
    = \,&\meanint_{B_r(x_0)}
    \bigg|\meanint_{B_r(x_0)}(v(x)-v(y))\,dy\bigg|^p\,dx \\
    \leq \,&\meanint_{B_r(x_0)}\bigg(\meanint_{B_r(x_0)}|v(x)-v(y)|^p\,dy\bigg)dx
    \\
    \leq\,& \big(\mathrm{osc}_{B_r(x_0)}v\big)^p \\[0.2cm]
     \leq\,& C\bigg(\frac{r}{R}\bigg)^{\alpha p}\bigg(
    \mathrm{Tail}(v,x_0,R)^p+\meanint_{B_{2R}(x_0)}|v|^p\,dx\bigg),
    \end{split}
   \end{equation}
   where $C$ and $\alpha$ are positive constants only
   depending on $n$, $s$ and~$p$. By combining 
   estimates
   \eqref{eq:estimfI}-\eqref{eq:estimsI}
   with \eqref{eq:toestimI1I2}, we then get
   \begin{equation} \label{eq:toestimTailmean}
   \begin{split}
    \meanint_{B_r(x_0)}
    |u-\overline{u}_{x_0,r}|^p\,dx 
    & \leq 
    C\,\bigg(\frac{R}{r}\bigg)^n\,R^{\gamma}
    \|f\|_{L^\infty(\Omega)}^{p'} \\
    & \qquad
    + C\bigg(\frac{r}{R}\bigg)^{\alpha p}\bigg(
    \mathrm{Tail}(v,x_0,R)^p+\meanint_{B_{2R}(x_0)}|v|^p\,dx\bigg),
   \end{split}
   \end{equation}
   where we have set
   \begin{equation}\label{eq:gamma}
   \gamma := np'-\frac{p'(n-sp)}{p}+sp-n >0.
   \end{equation}
   To complete the proof
   of \eqref{eq:excessdecay} we observe that,
    since $u \equiv v$ a.e.\,on $\R^n\setminus B_{3R}(x_0)$
   (and $0 < R \leq 1$), by definition of $\mathrm{Tail}(v,x_0,R)$
   we have
   \begin{equation} \label{eq:estimTailexcess}
   \begin{split}
    & \mathrm{Tail}(v,x_0,R)^p
    = R^p\int_{\R^n\setminus
  B_R(x_0)}\frac{|v|^{p}}{|x-x_0|^{n+ps}}\,dx \\
  & \qquad = R^p\int_{\R^n\setminus
  B_{4R}(x_0)}\frac{|v|^{p}}{|x-x_0|^{n+ps}}\,dx 
  + R^p\int_{
  B_{4R}(x_0)\setminus B_R(x_0)}\frac{|v|^{p}}{|x-x_0|^{n+ps}}\,dx \\[0.2cm]
  & \qquad \leq C\bigg(
  \mathrm{Tail}(u,x_0,4R)^p 
  + \meanint_{B_{4R}(x_0)}
  |v|^p\,dx\bigg).
   \end{split}
   \end{equation}
   Moreover, by using again Lemma \ref{lem:trequattroBrasco}, we get
   \begin{equation} \label{eq:estimmeanintexcess}
    \begin{split}
    \meanint_{B_{4R}(x_0)}
  |v|^p\,dx & \leq 
  C\meanint_{B_{4R}(x_0)}|u-v|^p\,dx + C\meanint_{B_{4R}(x_0)}|u|^p\,dx \\
  & \leq 
  C\bigg(R^{\gamma}\|f\|_{L^\infty(\Omega)}^{p'}
  + 
  \meanint_{B_{4R}(x_0)}|u|^p\,dx\bigg).
    \end{split}
   \end{equation}
  Thus, by inserting \eqref{eq:estimTailexcess}-\eqref{eq:estimmeanintexcess} 
  into \eqref{eq:toestimTailmean}, we obtain the desired
  \eqref{eq:excessdecay}.
  \end{proof}
  
  By combining Lemmata \ref{lem:trequattroBrasco} and~\ref{prop:decayestimate}, 
  we can provide the
  \begin{proof}[Proof of Theorem\,\ref{thm:mainHolder}]
   The proof follows the lines of \cite[Theorem 3.6]{BLS}. 
   First, we consider a ball $B_{R_0}(z) \subset\subset \Omega$ and we define the quantities
   \begin{equation}
   d:= \mathrm{dist}(B_{R_0}(z), \partial \Omega) >0
      \quad \textrm{and }
   \quad
   R_1:= \dfrac{d}{2}+R_0.
   \end{equation}
  Thus, we can choose a point $x_0 \in B_{R_0}(z)$ and the ball  $B_{4R}(x_0)$, where $R < \min\{1, \tfrac{d}{8}\}$. In particular, this implies that $B_{4R}(x_0)\subset B_{R_1}(z)$.
  Since $R<1$, we can then apply Lemma \ref{prop:decayestimate}: this gives,
  for every $0 < r \leq R$, 
  \begin{equation}\label{eq:StimaThm3.6}
  \begin{aligned}
  \meanint_{B_{r}(x_0)}&|u-\overline{u}_{x_0,r}|^p\,dx \leq C \left(\dfrac{R}{r}\right)^{n}R^{\gamma} \|f\|^{p'}_{L^{\infty}(\Omega)}\\
  &+ C\left(\dfrac{r}{R}\right)^{\alpha \,p}\left( R^{\gamma}\|f\|^{p'}_{L^{\infty}(\Omega)} + \meanint_{B_{4R}(x_0)}|u|^p \, dx + \mathrm{Tail}(u,x_0,4R)^p\right)\\
  &\leq C \left(\dfrac{R}{r}\right)^{n}R^{\gamma} \|f\|^{p'}_{L^{\infty}(\Omega)}\\
  &+ C\left(\dfrac{r}{R}\right)^{\alpha \,p}\left( d^{\gamma}\|f\|^{p'}_{L^{\infty}(\Omega)} + \|u\|_{L^{\infty}(\Omega)}^p \, dx + \mathrm{Tail}(u,x_0,4R)^p\right),
  \end{aligned}
  \end{equation}
  \noindent where $\gamma >0$ is as in \eqref{eq:gamma}.
Now, we notice that for every $x \notin B_{R_1}(z)$ it holds that
\begin{equation*}
|x-x_0| \geq |x-z|-|z-x_0| \geq \dfrac{R_1 - |z-x_0|}{R_1}|x-z|.
\end{equation*}
Therefore, we have
\begin{equation*}
\begin{split}
 \mathrm{Tail}(u,x_0,4R)^p 
 & = 
 (4R)^p \int_{\mathbb{R}^n \setminus B_{R_1}(z)} 
  \dfrac{|u|^p}{|x-x_0|^{n+ps}}\, dx \\
  & \qquad\qquad + (4R)^p \int_{B_{R_1}(z)\setminus B_{4R}(x_0)} 
  \dfrac{|u|^p}{|x-x_0|^{n+ps}}\, dx\\
 & \leq \left(\dfrac{4R}{R_1}\right)^{p}\left(\dfrac{R_1}{R_1 - |z-x_0|}\right)^{n+ps} 
 \mathrm{Tail}(u,z,R_1)^p + 
 C \|u\|^p_{L^{\infty}(\Omega)} \\[0.2cm]
 & \leq \mathrm{Tail}(u,z,R_1)^p + C \|u\|^p_{L^{\infty}(\Omega)}
\end{split}
\end{equation*}
\noindent for a constant $C$ depending on $n$, $s$ and~$p$. We recall that in the last estimate we exploited that
\begin{equation*}
\dfrac{4R}{R_1} < \dfrac{\tfrac{d}{2}}{R_0 + \tfrac{d}{2}}<1 \quad \textrm{and } \quad \dfrac{4R}{R_1 - |x_0-z|}\leq \dfrac{4R}{R_1-R_0}<1.
\end{equation*}
Consequently, continuing the estimate started with \eqref{eq:StimaThm3.6}, we find that
\begin{equation}\label{eq:StimaThm3.6Bis}
\begin{aligned}
  \meanint_{B_{r}(x_0)}&|u-\overline{u}_{x_0,r}|^p\,dx \leq C \left(\dfrac{R}{r}\right)^{n}R^{\gamma} \|f\|^{p'}_{L^{\infty}(\Omega)}\\
  &+ C\left(\dfrac{r}{R}\right)^{\alpha \,p}\left( d^{\gamma}\|f\|^{p'}_{L^{\infty}(\Omega)} + \|u\|_{L^{\infty}(\Omega)}^p \, dx + \mathrm{Tail}(u,z,R_1)^p\right).
  \end{aligned}
  \end{equation}
We can now define the positive number
$$\theta := 1 + \dfrac{\gamma}{n+\alpha \, p},$$
\noindent and take $r:= R^{\theta}$ in \eqref{eq:StimaThm3.6Bis}, which yields
\begin{equation*} 
\begin{split}
  & r^{-\beta p}\meanint_{B_{r}(x_0)\cap B_{R_0}(z)} |u-\overline{u}_{x_0,r}|^p\,dx  \\
  & \qquad \leq C \left( (d^{\gamma}+1)\|f\|^{p'}_{L^{\infty}(\Omega} + \|u\|^{p}_{L^{\infty}(\Omega)} + \mathrm{Tail}(u,z,R_1)^p\right),
  \end{split}
\end{equation*}
  where we have set 
  $$\beta:= \dfrac{\gamma \alpha}{n+\alpha p + \gamma}>0.$$ 
  This shows that $u \in \mathcal{L}^{p,n+\beta\gamma}(B_{R_0}(z))$, the 
  Campanato space isomorphic to the H\"{o}lder space $C^{0,\beta}(\overline{B_{R_0}(z)})$. 
  This completes the proof of Theorem~\ref{thm:mainHolder}.  
  \end{proof}
  
  By gathering together Theorems \ref{thm:GlobalBd} and 
  \ref{thm:mainHolder}, we can easily prove
  the needed \emph{interior H\"older regularity} of the eigenfunctions
  of $\LL_{p,s}$.
  
  \begin{theorem} \label{thm:contEigen}
   Let $\lambda\geq \lambda_1(\Omega)$ be an eigenvalue of
   $\LL_{p,s}$, and let $\phi_\lambda\in \mathcal{X}_0^{1,p}(\Omega)\setminus\{0\}$ 
   be an eigenfunction
   associated with $\lambda$. Then, $\phi_\lambda\in C(\Omega)$.
  \end{theorem}
  \begin{proof}
   On account of Theorem \ref{thm:GlobalBd}, we know that
   $\phi_\lambda\in L^\infty(\R^n)$. As a consequence, $\phi_\lambda$
   is a \emph{globally bounded weak solution}
   to \eqref{eq:mainPDE},
   with 
   $$f := \lambda|\phi_\lambda|^{p-2}\phi_\lambda\in L^\infty(\Omega).$$
   We are then entitled to apply Theorem \ref{thm:mainHolder}, which ensures
   that 
   $\phi_\lambda\in C^{0,\beta}_{\loc}(\Omega)$
   for some 
   $\beta = \beta(n,s,p)\in (0,1)$. This ends the proof of Theorem~\ref{thm:contEigen}.
  \end{proof}
 
 \section{The Hong-Krahn-Szeg\"o inequality for $\LL_{p,s}$}\label{LULT}
 In this last section
 of the paper we provide the proof of Theorem \ref{main:thm}.
 Before doing this, we establish two preliminary results.
 \medskip
 
 First of all, we prove 
 the following \emph{Faber-Krahn type inequality for $\LL_{p,s}$.}
 \begin{theorem} \label{thm.FK}
  Let~$\Omega\subseteq\R^n$ be a 
  bounded open set,
   and let $m:= |\Omega|\in (0,\infty)$. Then,
   if $B^{(m)}$ is any Euclidean ball with 
   volume
   $m$, one has
   \begin{equation} \label{eq.FK}
    \lambda_1(\Omega)\geq \lambda_1 (B^{(m)}).
   \end{equation}
   Moreover, if the equality holds in \eqref{eq.FK}, then 
   $\Omega$ is a ball.
   \end{theorem}
 \begin{proof}
  The proof is similar to that in the linear case, see \cite[Theorem\,1.1]{BDVV2};
  how\-e\-ver, we present it here in all the details for the sake of completeness.
  \vspace*{0.1cm}
  
   To begin with, let
   $\widehat{B}^{(m)}$ be the Euclidean ball with centre $0$ and volume $m$.
   Moreover, let $u_0\in\mathcal{M}(\Omega)$ be the
 principal
   eigenfunction for $\LL_{p,s}$. We recall that,
   by definition, $u_0$ is the unique non-negative eigenfunction
   associated with the first eigenvalue
   $\lambda_1(\Omega)$; in particular, we have (see \eqref{eq:uzerolambda1})
   \begin{equation} \label{eq:lambda1uzeroFK} 
    \lambda_1(\Omega) =  
  \int_\Omega|\nabla u_0|^p\,dx + \iint_{\R^{2n}}
 \frac{|u_0(x)-u_0(y)|^p}{|x-y|^{n+ps}}\,dx\,dy.
 \end{equation}
   Then, we define
   $u_0^\ast:\R^n\to\R$
   as the (decreasing) Schwarz symmetrization of $u_0$.
   Now, since $u_0\in\mathcal{M}(\Omega)$, from
   the well-known inequality by P\`olya and Szeg\"o (see e.g.~\cite{PSZ})
   we deduce that
   \begin{equation} \label{eq.PSLoc}
    u_0^\ast\in \mathcal{M}(\widehat{B}^{(m)})\qquad\text{and}\qquad
   \int_{\widehat{B}^{(m)}}|\nabla u_0^\ast|^p\, dx
   \leq \int_{\Omega}|\nabla u|^p\, dx.
   \end{equation}
Furthermore, by \cite[Theorem\,9.2]{AlLieb}
   (see also \cite[Theorem\,A.1]{FS}), we also have
   \begin{equation} \label{eq.PSNonloc}
    \iint_{\R^{2n}}\frac{|u_0^\ast(x)-u_0^\ast(y)|^p}{|x-y|^{n+ps}}\, d x\, dy
   \leq \iint_{\R^{2n}}\frac{|u_0(x)-u_0(y)|^p}{|x-y|^{n+ps}}\, d x\, dy.
   \end{equation}
   Gathering all these facts and using \eqref{eq:lambda1uzeroFK}, we get
   \begin{equation} \label{eq.estimuusharp}
   \begin{split}
    \lambda_{1}(\Omega)
    & 
    = \int_{\Omega}|\nabla u_0|^2\, dx
    +\iint_{\R^{2n}}\frac{|u_0(x)-u_0(y)|^2}{|x-y|^{n+2s}}\, d x\, dy \\
    & \geq 
    \int_{\widehat{B}^{(m)}}|\nabla u_0^\ast|^2\, dx
    + \iint_{\R^{2n}}\frac{|u_0^\ast(x)-u_0^\ast(y)|^2}{|x-y|^{n+2s}}\, d x\, dy\\&
    \geq \lambda_{1}(\widehat {B}^{(m)}).
   \end{split}
   \end{equation}
   From this, since $\lambda_{1}(\cdot)$ is translation-invariant, 
   we derive the validity
   of \eqref{eq.FK} for every Euclidean ball $B^{(m)}$ with volume $m$.
   
   To complete the proof of Theorem~\ref{thm.FK}, let us suppose that 
   $$\lambda_{1}(\Omega) = \lambda_{1}(B^{(m)})$$
   for some (and hence, for every) ball $B^{(m)}$ with $|B^{(m)}| = m$.
   By \eqref{eq.estimuusharp} we have
   \begin{align*}
    & \int_{\Omega}|\nabla u_0|^p\, dx
    +\iint_{\R^{2n}}\frac{|u_0(x)-u_0(y)|^p}{|x-y|^{n+ps}}\, d x\, dy
    = \lambda_1(\Omega) \\
    & \qquad = \lambda_1(\widehat{B}^{(m)}) =
    \int_{\widehat{B}^{(m)}}|\nabla (u_0)^\ast|^p\, dx
    + \iint_{\R^{2n}}\frac{|u_0^\ast(x)-u_0^\ast(y)|^p}{|x-y|^{n+ps}}\, d x\, dy.
    \end{align*}
    In particular, by \eqref{eq.PSLoc} and~\eqref{eq.PSNonloc} we get
    $$
    \iint_{\R^{2n}}\frac{|u_0(x)-u_0(y)|^p}{|x-y|^{n+ps}}\, d x\, dy 
    = \iint_{\R^{2n}}\frac{|u_0^\ast(x)-u_0^\ast(y)|^p}{|x-y|^{n+ps}}\, d x\, dy.$$
    We are then in the position to apply once again \cite[Theorem\,A.1]{FS},
    which ensures that
    $u_0$ must be proportional to a translation of a symmetric decreasing function.
    As a consequence of this fact, we immediately deduce that
    $$\Omega = \{x\in\R^n:\,u_0(x) > 0\}$$
    must be a ball (up to a set of zero Lebesgue measure).  This completes the proof of Theorem~\ref{thm.FK}.
\end{proof}

 Then, we establish the following lemma on \emph{nodal domains}.
 
 \begin{lemma} \label{lem:nodal}
 Let $\lambda>\lambda_1(\Omega)$ be an eigenvalue of $\LL_{p,s}$,
 and let $\phi_\lambda\in \mathcal{X}_0^{1,p}(\Omega)\setminus\{0\}$
 be an eigenfunction associated with $\lambda$. We define the sets
 \begin{equation*}
 \Omega^+ := \left\{ x \in \Omega: \phi_{\lambda}(x) >0\right\} \quad \textrm{and} \quad \Omega^- := \left\{ x \in \Omega: \phi_{\lambda}(x) <0\right\}.
 \end{equation*}
 Then $\lambda > \max\left\{\lambda_1(\Omega^{+}), \lambda_1(\Omega^-)\right\}$.
 \end{lemma}
  The proof of Lemma \ref{lem:nodal} 
  takes inspiration from 
   \cite[Lemma\,6.1]{BP} (see also \cite[Lemma\,4.2]{GoelSreenadh}).
   
 \begin{proof}[Proof of Lemma~\ref{lem:nodal}]
 First of all, on account of Theorem \ref{thm:contEigen}
 we have that the sets~$\Omega^+$ and~$\Omega^-$ are open, and 
 therefore the eigenvalues~$\lambda_{1}(\Omega^{\pm})$ are well--defined. 
 
 Moreover, thanks to Proposition \ref{prop.BMV}, 
 we know that $\phi_\lambda$ changes sign in $\Omega$, and therefore it is convenient to write
 $\phi_{\lambda}= \phi_{\lambda}^+ - \phi_{\lambda}^-,$
 where $\phi_{\lambda}^+$ and $\phi_{\lambda}^-$ denote, respectively, the positive and negative parts of $\phi_{\lambda}$, with the convention that both the functions~$\phi_{\lambda}^+$ and $\phi_{\lambda}^-$ are non-negative.
 
 Let us now start proving that $\lambda > \lambda_{1}(\Omega^+)$.
 By using the fact that $\phi_{\lambda}$ is an eigenfuction of $\LL_{p,s}$ corresponding to $\lambda$, it follows that
 \begin{equation*}
 \begin{aligned}
 & \int_{\Omega}|\nabla \phi_{\lambda}|^{p-2}\langle \nabla \phi_{\lambda},v\rangle \, dx \\
 & \qquad\qquad 
 + \iint_{\mathbb{R}^{2n}}\dfrac{|\phi_{\lambda}(x)-\phi_{\lambda}(y)|^{p-2}(\phi_{\lambda}(x)-\phi_{\lambda}(y))(v(x)-v(y)}{|x-y|^{n+ps}}\, dxdy\\
 & \qquad = \lambda \int_{\Omega}|\phi_{\lambda}|^{p-2}\phi_{\lambda}v \, dx, \quad \textrm{ for all } v \in \mathcal{X}_{0}^{1,p}(\Omega).
 \end{aligned}
 \end{equation*}
 In consideration of the fact that $\phi_{\lambda}^{+} \in \mathcal{X}_{0}^{1,p}(\Omega)$, 
 we can take $v = \phi_{\lambda}^+$ as a test function. 
 
 Now, since 
 $$\text{$\phi_{\lambda}^{+}(x)\phi_{\lambda}^{-}(x)=0$ for a.e. $x \in \Omega$},$$ 
 we easily get that
$$ (\phi_{\lambda}^+ (x) - \phi_{\lambda}^{+}(y))(\phi_{\lambda}^- (x) - \phi_{\lambda}^{-}(y))\leq 0.$$
 Moreover, since both $\Omega_+$ and $\Omega_-$ are non-void open set
(remind that $\phi_\lambda$ is con\-ti\-nuo\-us on $\Omega$ and it changes sign in $\Omega$), we have
\begin{eqnarray*} 
 &&\iint_{\R^{2n}}\frac{|\phi_\lambda(x)-\phi_\lambda(y)|^{p-2}
(\phi_{\lambda}^+ (x) - \phi_{\lambda}^{+}(y))(\phi_{\lambda}^- (x) - \phi_{\lambda}^{-}(y))}
{|x-y|^{n+ps}}\,d x\,d y \\
&& \qquad\qquad\qquad\leq -\int_{\Omega_+}\int_{\Omega_-}
\frac{|\phi_\lambda(x)-\phi_\lambda(y)|^{p-2}
\phi_{\lambda}^+ (x)\phi_{\lambda}^{-}(y)}
{|x-y|^{n+ps}}\,d x\,d y < 0\end{eqnarray*}
and
\begin{eqnarray*} &&\iint_{\R^{2n}}\frac{|\phi_\lambda^+(x)-\phi_\lambda^+(y)|^{p-2}
(\phi_{\lambda}^+ (x) - \phi_{\lambda}^{+}(y))(\phi_{\lambda}^- (x) - \phi_{\lambda}^{-}(y))}
{|x-y|^{n+ps}}\,d x\,d y \\
&& \qquad\qquad\qquad\leq -\int_{\Omega_+}\int_{\Omega_-}
\frac{|\phi^+_\lambda(x)|^{p-2}
\phi_{\lambda}^+ (x)\phi_{\lambda}^{-}(y)}
{|x-y|^{n+ps}}\,d x\,d y < 0.
\end{eqnarray*}
 We can therefore exploit Lemma \ref{lem:algebrico}-(1) with 
 $$a := \phi_{\lambda}^+ (x) - \phi_{\lambda}^{+}(y)\qquad\text{and}\qquad
  b := \phi_{\lambda}^- (x) - \phi_{\lambda}^{-}(y),$$ 
  obtaining (remind that, by assumption, $p\geq 2$)
  \begin{equation*}
  \begin{aligned}
  &\lambda \int_{\Omega^+}|\phi_{\lambda}^+|^{p}\, dx =
  \lambda \int_{\Omega}|\phi_{\lambda}|^{p-2}\phi_{\lambda}\phi_{\lambda}^+ \, dx \\
  &\qquad 
  =\int_{\Omega}\|\nabla \phi_{\lambda}\|^{p-2}\langle \nabla \phi_{\lambda}, \nabla \phi_{\lambda}^{+}\rangle \, dx \\
  & \qquad\qquad 
  + \iint_{\mathbb{R}^{2n}}\dfrac{|\phi_{\lambda}(x)-\phi_{\lambda}(y)|^{p-2}(\phi_{\lambda}(x)-\phi_{\lambda}(y))(\phi_{\lambda}^{+}(x)-\phi_{\lambda}^{+}(y)}{|x-y|^{n+ps}}\, dxdy\\
  &\qquad = \int_{\Omega^+}\|\nabla \phi_{\lambda}^{+}|^{p}\, dx \\
  & \qquad\qquad + \iint_{\mathbb{R}^{2n}}\dfrac{|\phi_{\lambda}(x)-\phi_{\lambda}(y)|^{p-2}(\phi_{\lambda}(x)-\phi_{\lambda}(y))(\phi_{\lambda}^{+}(x)-\phi_{\lambda}^{+}(y)}{|x-y|^{n+ps}}\, dxdy\\
  &\qquad > \int_{\Omega^+}|\nabla \phi_{\lambda}^{+}|^{p}\, dx + \iint_{\mathbb{R}^{2n}}\dfrac{|\phi_{\lambda}^{+}(x)-\phi_{\lambda}^{+}(y)|^p}{|x-y|^{n+ps}}\, dxdy \\
  & \qquad \geq \lambda_{1}(\Omega^+)  \int_{\Omega^+}|\phi_{\lambda}^+|^{p}\, dx,
  \end{aligned}
  \end{equation*}
  \noindent where we used the variational characterization
  of  $\lambda_{1}(\Omega^+)$, see \eqref{eq:deflambdavar}. 
  In particular, this gives that $\lambda > \lambda_{1}(\Omega^+)$. With a similar argument (see 
  e.g. \cite[Lemma\,6.1]{BP}), one can show that $\lambda > \lambda_{1}(\Omega^-)$ as well, and this closes 
  the proof of Lemma~\ref{lem:nodal}. 
 \end{proof}
 By virtue of Theorem \ref{thm.FK} and Lemma \ref{lem:nodal}, we can provide the
 \begin{proof}
 [Proof of Theorem\,\ref{main:thm}]
 We split the proof into two steps.
 \medskip
 
 \textsc{Step I:} In this step, we prove inequality 
 \eqref{eq:HKS}. 
 To this end, let $\phi\in\mathcal{M}(\Omega)$ be a $L^p$-normalized 
 eigenfunction
 associated with $\lambda_2(\Omega)$ (recall the definition of the space~$\mathcal{M}(\Omega)$ in~\eqref{eq:defM}). 
 On account of Theorem \ref{thm:GS}, we know that
 $\phi\in C(\Omega)$.
 
 Moreover, since $\phi$ changes sign in $\Omega$
 (see Proposition \ref{prop.BMV}), we can define the non-void open sets
 $$\Omega_+ := \{u > 0\}\qquad\text{and}\qquad
 \Omega_- := \{u < 0\}.$$
 Then, by combining Lemma \ref{lem:nodal} with
 Theorem \ref{thm.FK}, we get
 \begin{equation} \label{eq:lambda2max}
 \lambda_2(\Omega) > \max\big\{\lambda_1(B_+),\lambda_1(B_-)\big\},
 \end{equation}
 where $B_+$ is a Euclidean ball with volume equal to $|\Omega_+|$ and
 $B_-$ is a Euclidean ball with volume $|\Omega_-|$.
 
 Now, since $\Omega_+\cup\Omega_- = \Omega$, we have
 $$|B_+|+|B_-| = |\Omega_+|+|\Omega_-| \leq |\Omega| = m.$$
Taking into account this inequality, we claim that
 \begin{equation} \label{eq:maxgeqBall}
  \max\big\{\lambda_1(B_+),\lambda_1(B_-)\big\} \geq \lambda_1(B),
 \end{equation}
 being~$B$ a ball of volume~$|\Omega|/2$.
 In order to prove \eqref{eq:maxgeqBall}, we distinguish three cases.
 \begin{itemize}
  \item[(i)] $|B_+|,\,|B_-|\leq m/2$. In this case, since 
 $\lambda_1(\cdot)$ is translation-invariant, 
 we can assume without loss of generality that
 $B\subseteq B_+,\,B_-$;
 as a consequence, since $\lambda_1(\cdot)$ is non-increasing, we obtain
 $$\lambda_1(B_+),\,\lambda_1(B_-)\geq \lambda_1(B),$$
 and this proves the claimed \eqref{eq:maxgeqBall}.
 \vspace*{0.2cm}
 
 \item[(ii)] $|B_-| < m/2 < |B_+|$. In this case, we can assume that
 $B_-\subseteq B\subseteq B_+$;
 from this, since $\lambda_1(\cdot)$ is non-increasing, we obtain
 $$\lambda_1(B_+)\geq \lambda_1(B)\geq \lambda_1(B_-),$$
 and this immediately implies the claimed \eqref{eq:maxgeqBall}.
 \vspace*{0.2cm}
 
 \item[(iii)] $|B_+| < m/2 < |B_-|$. In this last case, 
 it suffices to interchange the r\^oles of the balls
 $B_-$ and $B_+$, and to argue
 exactly as in case~(ii).
 \end{itemize}
 Gathering \eqref{eq:lambda2max} and 
 \eqref{eq:maxgeqBall}, we obtain the claim in~\eqref{eq:HKS}.
 \medskip
 
 \textsc{Step II:} Now we prove the sharpness
 of \eqref{eq:HKS}. To this end, according to
 the statement of the theorem, we
 fix $r > 0$ and we define
 $$\Omega_j := B_r(x_j)\cup B_r(y_j),$$
 where $\{x_j\}_j,\,\{y_j\}_j\subseteq\R^n$ are two sequences satisfying
 \begin{equation} \label{eq:xjyjdiv}
  \lim_{j\to+\infty}|x_j-y_j| = +\infty.
 \end{equation}
 On account of \eqref{eq:xjyjdiv}, we can assume that 
 \begin{equation} \label{eq:disjoint}
  B_r(x_j)\cap B_r(y_j) = \varnothing\qquad{\mbox{ for all }} j\geq 1.
 \end{equation}
 Let now $u_0\in \mathcal{M}(B_r)$ be a $L^p$-normalized eigenfunction associated
 with $\lambda_1(B_r)$ (here, $B_r = B_r(0)$). For every natural number $j\geq 1$, we set
 \begin{equation} \label{eq:defphiij}
  \phi_{j}(x) := u_0(x-x_j)\qquad\text{and}\qquad
 \psi_{j}(x) := u_0(x-y_j).
 \end{equation}
 Since $\lambda_1(\cdot)$ is translation-invariant, it is immediate to check that
 $\phi_j$ and $\psi_j$ are normalized eigenfunctions
 associated with $\lambda_1(B_r(x_j))$ and $\lambda_1(B_r(y_j))$,
 respectively.
 
 Moreover,
 taking into account 
 \eqref{eq:disjoint}, it is easy to see that
\begin{equation}\label{propa}
{\mbox{$\phi_j\equiv 0 $ on $\R^n\setminus B_r(x_j)\supseteq B_r(y_j)$
  and
   $\psi_j\equiv 0$ on $\R^n\setminus B_r(y_j)\supseteq B_r(x_j)$}}\end{equation}
and~$\phi_j\psi_j\equiv 0$ on $\R^n$.

 We then consider the function $f$ defined as follows:
 $$
 f(z_1,z_2) := |z_1|^{\frac{2-p}{p}}z_1\phi_{j}-
 |z_2|^{\frac{2-p}{p}}z_2\psi_j\qquad
 \text{with $z = (z_1,z_2)\in S^1$}.$$
 Taking into account that $B_r(x_j),\,B_r(y_j)\subseteq\Omega_j$
 and $u_0\in \mathcal{M}(B_r)$, it is readily seen that
 $f(S_1)\subseteq \mathcal{X}_0^{1,p}(\Omega_j)$.
 
 Furthermore, the function $f$
 is clearly \emph{odd} and continuous.
 Also, using \eqref{eq:disjoint} and the fact that
 $\phi\equiv 0$ out of $B_r$, one has
 \begin{align*}
  \|f(z_1,z_2)\|^p_{L^p(\Omega_j)} & = \big\||z_1|^{\frac{2-p}{p}}z_1\phi_{j}-
 |z_2|^{\frac{2-p}{p}}z_2\psi_j\big\|^p_{L^p(\Omega_j)}  \\
 &  = |z_1|^2\|\phi_j\|^p_{L^p(B_r(x_j))}
 + |z_2|^2\|\psi_j\|^p_{L^p(B_r(y_j))} \\
 &  =
 (|z_1|^2+|z_2|^2)\|u_0\|_{L^p(B_r)}^p \\&= 1.
 \end{align*}
 We are thereby entitled to use $f$ in the definition
 of $\lambda_2(\Omega)$, see \eqref{eq:deflambda2}:
  setting $a_j:= \phi_j(x)-\phi_j(y)$ and $b_j:= \psi_j(x)-\psi_j(y)$ to simplify the notation,
 this gives, together with~\eqref{eq:disjoint} and~\eqref{propa}, that
 \begin{align*}
  \lambda_2(\Omega_j)
 & \leq 
 \max_{v\in\mathrm{Im}(f)}
  \bigg\{\int_{\Omega_j}
  |\nabla v|^p\,dx  
   +
  \iint_{\R^{2n}}\frac{|v(x)-v(y)|^p}{|x-y|^{n+ps}}
  \,dx\,dy\bigg\} \\
  & 
  = \max_{|\omega_1|^p+|\omega_2|^p = 1}
  \bigg\{\int_{\Omega_j}
  |\nabla(\omega_1\phi_{j}-
 \omega_2\psi_j)|^{p}\,dx  +
  \iint_{\R^{2n}}\frac{|\omega_1a_j-\omega_2b_j
  |^p}{|x-y|^{n+ps}}
  \,dx\,dy\bigg\} \\
  & = 
  \max_{|\omega_1|^p+|\omega_2|^p = 1}
  \bigg\{|\omega_1|^p\int_{B_r(x_j)}|\nabla\phi_j|^p\,dx
  + |\omega_2|^p\int_{B_r(y_j)}|\nabla\psi_j|^p\,dx \\
  & \qquad\qquad\qquad\qquad
  + \iint_{\R^{2n}}\frac{|\omega_1a_j-\omega_2b_j
  |^p}{|x-y|^{n+ps}}
  \,dx\,dy\bigg\} \\
  & = 
  \max_{|\omega_1|^p+|\omega_2|^p = 1}
  \bigg\{\int_{B_r}|\nabla u_0|^p\,d x
  + \iint_{\R^{2n}}\frac{|\omega_1a_j-\omega_2b_j
  |^p}{|x-y|^{n+ps}}\bigg\} .
 \end{align*}
 On the other hand, by applying Lemma \ref{lem:algebrico}-(2), we get
 \begin{align*}
    & \max_{|\omega_1|^p+|\omega_2|^p = 1}
  \bigg\{\int_{B_r}|\nabla u_0|^p\,d x
  + \iint_{\R^{2n}}\frac{|\omega_1a_j-\omega_2b_j
  |^p}{|x-y|^{n+ps}}\bigg\}\\& \leq 
  \max_{|\omega_1|^p+|\omega_2|^p = 1}
  \bigg\{\int_{B_r}|\nabla u_0|^p\,d x
  + |\omega_1|^p\iint_{\R^{2n}}\frac{|\phi_j(x)-\phi_j(y)|^p}{|x-y|^{n+ps}}\,dx\,d y
  \\
  & \qquad\qquad
  + |\omega_2|^p\iint_{\R^{2n}}\frac{|\psi_j(x)-\psi_j(y)|^p}{|x-y|^{n+ps}}\,dx\,d y
  \\
  & \qquad\qquad+
  c_p\iint_{\R^{2n}}\frac{|(\omega_1a_j)^2+(\omega_2b_j)^2|^{\frac{p-2}{2}}|\omega_1\omega_2a_jb_j|}
  {|x-y|^{n+ps}}\,dx\,d y\bigg\} \\
  & =
  \max_{|\omega_1|^p+|\omega_2|^p = 1}
  \bigg\{\int_{B_r}|\nabla u_0|^p\,d x
  + \iint_{\R^{2n}}\frac{|u_0(x)-u_0(y)|^p}{|x-y|^{n+ps}}\,dx\,d y
  \\
  & \qquad\qquad+
  c_p\iint_{\R^{2n}}\frac{|(\omega_1a_j)^2+(\omega_2b_j)^2|^{\frac{p-2}{2}}|\omega_1\omega_2a_jb_j|}
  {|x-y|^{n+ps}}\,dx\,d y\bigg\}
  \\
  & = \lambda_1(B_r)+c_p\,
  \max_{|\omega_1|^p+|\omega_2|^p}\iint_{\R^{2n}}\frac{|(\omega_1a_j)^2+(\omega_2b_j)^2|^{\frac{p-2}{2}}|\omega_1\omega_2a_jb_j|}
  {|x-y|^{n+ps}}\,dx\,d y,
 \end{align*}
where we have also used that $u_0$ is a normalized eigenfunction associated with
the first eigenvalue $\lambda_1(B_r)$.
 
 Summarizing, we have proved that
 \begin{equation} \label{eq:topasslimit}
 \begin{split}
  & \lambda_2(\Omega_j) \leq \lambda_1(B_r) \\
  & \qquad\qquad
  +
  c_p\,
  \max_{|\omega_1|^p+|\omega_2|^p}\iint_{\R^{2n}}\frac{|(\omega_1a_j)^2+(\omega_2b_j)^2|^{\frac{p-2}{2}}|\omega_1\omega_2a_jb_j|}
  {|x-y|^{n+ps}}\,dx\,d y.
  \end{split}
 \end{equation}We now set
 $$\mathcal{R}_j := \max_{|\omega_1|^p+|\omega_2|^p=1}\iint_{\R^{2n}}
 \frac{|(\omega_1a_j)^2+(\omega_2b_j)^2|^{\frac{p-2}{2}}|\omega_1\omega_2a_jb_j|}
  {|x-y|^{n+ps}}\,dx\,d y$$
 and we claim that $\mathcal{R}_j\to 0$ as $j\to+\infty$.
 
 Indeed, since
 $\phi_j\psi_j\equiv 0$ on $\R^n$, we have that
 $$
 a_jb_j = -\psi_j(x)\psi_j(y) -\phi_j(y)\psi_j(x).$$
 As a consequence, recalling~\eqref{propa}
 \begin{align*}
  0\leq \mathcal{R}_j
  & 
  \leq 
  2\max_{|\omega_1|^p+|\omega_2|^p=1}\int_{B_r(x_j)}\int_{B_r(y_j)}
  \frac{|(\omega_1a_j)^2+(\omega_2b_j)^2|^{\frac{p-2}{2}}|\omega_1\omega_2a_jb_j|}
  {|x-y|^{n+ps}}\,dx\,d y \\
  &
  \leq \frac{2}{|x_j-y_j|-2r} \\
  & \qquad\times\max_{|\omega_1|^p+|\omega_2|^p = 1}\int_{B_r(x_j)}\int_{B_r(y_j)}
  |(\omega_1a_j)^2+(\omega_2b_j)^2|^{\frac{p-2}{2}}|\omega_1\omega_2a_jb_j|\,dx\,d y \\
  & = \frac{2}{|x_j-y_j|-2r}
  \int_{B_r}\int_{B_r}
  |u_0(x)^2+u_0(y)^2|^{\frac{p-2}{2}}
  |u_0(x)u_0(y)|\,dx\,d y  \\[0.1cm]
  & =: \frac{2c_0}{|x_j-y_j|-2r}.
 \end{align*}
 Taking into account \eqref{eq:xjyjdiv}, we thereby conclude that
 \begin{equation} \label{eq:limRjzero}
  \lim_{j\to+\infty}\mathcal{R}_j = 0.
 \end{equation}
 Gathering together \eqref{eq:topasslimit} and 
 \eqref{eq:limRjzero}, we obtain the desired result in~\eqref{eq:optimal}.
 \end{proof}

\vfill
\end{document}